\documentclass[final]{aupl}

\newcommand{\labbel}[1]{\label{#1} [[{\bf #1}]]}  
\newcommand{\bibbitem}[1]{\bibitem{#1} [[{\bf #1}]]}  
\renewcommand{\labbel}{\label} \renewcommand{\bibbitem}{\bibitem}

\usepackage{
amsfonts,
latexsym,
amssymb,
amsmath,
amsthm,
color,
enumerate,
verbatim}

\usepackage[matrix,arrow]{xy}

\numberwithin{equation}{section}

\newtheorem{theorem}{Theorem}[section]
\newtheorem{lemma}[theorem]{Lemma}

\newtheorem*{claim*}{Claim}

\newtheorem*{theorem*}{Theorem}
\newtheorem*{corollary*}{Corollary}

\theoremstyle{definition}
\newtheorem{definition}[theorem]{Definition}

\theoremstyle{remark}
\newtheorem{remark}[theorem]{Remark}

\begin{document}

\title{One-sided inverses in noncommutative infinitary semigroups}

\author{Paolo Lipparini} 
\address{
Dipartimento di Matematica. Viale della Ricerca Scientifica da un lato,
II Universit\`a di Roma (Tor Vergata), I-00133 ROME, ITALY (currently retired)
\\ORCiD 0000-0003-3747-6611}
\email{lipparin@axp.mat.uniroma2.it}

\thanks{Work performed under the auspices of G.N.S.A.G.A. Work 
 supported by PRIN 2012 ``Logica, Modelli e Insiemi''.
The author acknowledges the MIUR Department Project awarded to the
Department of Mathematics, University of Rome Tor Vergata, CUP
E83C18000100006.}

\begin{abstract}
In a former paper we introduced \emph{partial infinitary 
noncommutative semigroups} and showed, among other,
that significant
differences arise in comparison with the 
 commutative case, previously studied in the literature. For example,
in the commutative case we cannot have an infinitary identity $e$ 
together with two elements $a \not= e$, $b \not= e$ such that $ab= e$,
just under the assumption that the countable product $abababa\dots$ is defined.
Here we show that this is possible in the noncommutative case, actually, 
we can have an
 infinitary semigroup on a countable set with
a complete identity and such that the operation
is  defined for every indexed linearly ordered set.  
\end{abstract}



\maketitle

\section{Introduction} \labbel{intr} 

In \cite{smgrnc}  we introduced and studied the notion of a
``noncommutative''  
 \emph{infinitary 
semigroup}. We have dealt with the case in which the infinitary
operation is partial and also with the case when  the infinitary 
semigroup is \emph{complete}, namely, every product indexed
by a linearly ordered sequence is defined. 
The definitions generalize and encompass many 
previous notions introduced by various authors, including 
A. Tarski \cite{Tord}, C. Karp \cite{Karp}, 
J. H. Conway \cite{C}, D. Krob \cite{K},
N. Bedon  
 and C. Rispal \cite{BR},
among many others. See for example 
 \cite{HW}, \cite{PP} for a survey.

An easy argument from \cite[Proposition 5.1]{smgrnc}  shows that
we cannot have an infinitary identity 
$e$ together with two elements $a$, $b$   
such that $ba=ab=e$, apart from the trivial case
$a=b=e$. 
There we stated that, instead, we can have a 
complete infinitary semigroup with a complete identity
if we remove the request $ba=e$, namely, we just require
the existence of a one-sided inverse.
Since the proof turned out to be unexpectedly long
and complicated,
we decided to present it in a separate note.
Here it is.
This shows that the noncommutative theory
turns out to be much richer than  the commutative one,
though sometimes also much more difficult.

Our main result here is the following theorem 
(undefined notions will be explained in the next section).

\begin{theorem} \labbel{abe}  
There is  a  complete semigroup $(S, \prod)$ 
such that $S$ is a countable set and there are
elements $a,b, e \in S$  such that 
$ a  \neq e$, $ b \neq e$,
$ab=e$ and
 $e$ is a complete identity. 
\end{theorem}  

The proof of Theorem \ref{abe} will be given
in Section \ref{proofabe}.

\section{Preliminaries} \labbel{prel} 

\begin{definition} \labbel{pnc}
 An  \emph{infinitary semigroup}
 is a nonempty class $S$ 
together with a (partial) class function $ \prod$
whose codomain is $S$ itself and
whose domain
 consists of a class of linearly-indexed sequences
 of elements of $S$.
The image of $(a_i) _{i \in I} $  under $\prod$, when defined,
will be denoted by $ \prod _{i \in I} a_i $.
An infinitary 
semigroup is required to satisfy the following two properties.
  \begin{enumerate}
\item[(U)]
\emph{If $I= \{ i  \} $ has one element,
then 
$\prod _{i \in I} a_i $ is defined 
and is equal to $a_i$.}
   \item[(N)]
\emph{Whenever $\prod _{i \in I} a_i $ is defined
and $\pi:I \to J$ is a surjective order preserving map,
then all products in the following equation are defined
\begin{equation*}\labbel{N}
\prod _{i \in I} a_i = \prod _{j\in J} \prod _{\pi(i) =j }  a_i
  \end{equation*}     
and equality actually holds.}
   \end{enumerate} 

In the above-displayed formula 
the meaning of $ \prod _{\pi(i) =j }  a_i$
is $ \prod _{i \in I_j }  a_i$, where
$I_j  = \{ i \in I \mid \pi(i) =j \}  $ is given the (sub-)order 
induced by the order on $I$. Note that, since $\pi$ is  order preserving, 
$I_j$ is a \emph{convex} subset of $I$, 
that is, if $a <b \in I_j$ and $a< c<b$ holds in $I$, then
$c \in I_j$. 

If $a_i=a$, for every $i,j \in I$,
we will simply write   $ \prod _{I} a $,
and $\prod _ \emptyset $ in the case of the empty sequence.
As usual, when there is no risk of ambiguity, we will  write
$S$ in place of $(S, \prod)$.

A \emph{complete  semigroup} (a \emph{${\leq} \omega$-semigroup})
 is an infinitary semigroup such that  
$\prod _{i \in I} a_i $ is defined for every nonempty  linearly-indexed 
sequence
(respectively, nonempty well-ordered sequence  of type $\leq \omega $) 
of elements of $S$.
\end{definition}

It is immediate from  (N) and (U) that 
every partial infinitary semigroup satisfies the following property.
  \begin{enumerate} 
\item[(Iso)]
\emph{If 
$I$ and $J$ are isomorphic ordered set,
$f:I \to J$ is an order isomorphism
and $b _{f(i)}  = a_i$, for every $i \in I$, 
then 
$\prod _{i \in I} a_i $ is defined 
if and only if 
$\prod _{j \in J} b_j $ is defined and, in case they are defined,
they are equal.} 
  \end{enumerate}

In particular, if $I= \{ i_1, i_2\dots i_n \} $
 is finite, we are free to write,  without ambiguity,
$a_{i_1}a_{i_2}\dots a_{i_n} $ 
in place of  $\prod _{i \in I} a_i $.
We will use a similar notation also for longer sequences,
say, $a_{1}a_{2} a_{3}\dots $ for $\prod _{i \in \mathbb N} a_i $.

\begin{definition} \labbel{complid}    
If $S$ is an infinitary semigroup, 
an element $e $ of $  S$ is a \emph{complete identity}  
if it satisfies the following condition.
  \begin{enumerate}
   \item[(Id)]
\emph{If $\prod _{i \in I} a_i $ is defined
and  the (not necessarily convex) subset $H = \{ i \in I \mid a_i \not= e \} $
of $I$ has the induced order, 
then $\prod _{i \in H} a_i $ is defined, too, and
$\prod _{i \in I} a_i = \prod _{i \in H} a_i $.}
  \end{enumerate}

In particular, we get from  (U) that if $e$ is a complete identity, then
 $\prod _{\emptyset } $ is defined and equal to $e $.
Thus if some complete identity exists,  it is unique.
\end{definition}

\section{Sketch of a simpler example} \labbel{XyX}

We first exemplify our proof by sketching the construction of a
$\leq \omega $-semigroup $S$ with  elements $a,b,e$
satisfying the conclusions of Theorem \ref{abe}. 
The argument deals with strings; here juxtaposition denotes
string concatenation; in case of ambiguity we will always
specify whether we are dealing with string concatenation or
with a $\prod$-product in some infinitary semigroup. 

First, if $\sigma$ is either a finite string, or an $ \omega$-indexed string
and $\sigma$ 
contains only the symbols $a$ and $b$,
let the \emph{reduced form}  of $\sigma$ be
obtained from $\sigma$ by computing 
$\sigma$  ``as if we were in a semigroup in the 
ordinary sense'', with the clause $ab=e$, $e$ being
a neutral element. Namely, we 
iteratively remove all adjacent  pairs of the form $ab$.
This might involve removing infinitely
many elements in an infinite sequence; for example, if  
$\sigma=bba
\overline{a \overline{ab} b} \hspace{1.5 pt}
\overline{a\overline{ab}\hspace{1.5 pt}\overline{ab}
\hspace{1.5 pt}\overline{ab}b} \overline{ab}
 aa 
 \overline{ab}  \hspace{1.5 pt}
\overline{a \overline{ab} b} \hspace{1.5 pt}
\overline{a \overline{a \overline{ab} b} b} \hspace{1.5 pt}
 \dots$,
then the reduced form of $\sigma$ is $bbaaa$. 
The reduced form of a string might be infinite,
for example $bbbbbaaaaaaa\dots$ is already in reduced form. 
Of course, no occurrence of the substring $ab$
can appear in the reduced
form of $\sigma$, but it does not seem to be
immediate that the reduction is unique. 
See Appendix I.

Now we construct our first example which is 
a $\leq \omega $-semigroup.
The domain
$S$ is the set of all the finite strings of the form 
$b^na^m$, with $n,m \geq 0$, plus an additional element $ \Omega$.  
We denote the empty string, that is, the case $n=m=0$ by $e$. 
If $(\sigma_i) _{i \in I } $ is a sequence of elements of $S$
with either $I$ finite or $I $ of order-type $\omega $,
let us compute $\prod _{i \in I}  \sigma _i $ in the following way.

If some $\sigma_i$ is $ \Omega$, we let  $\prod _{i \in I}  \sigma _i  = \Omega $.

Otherwise, consider the concatenation
$ \sigma  = \ast _{ i \in I }  \sigma _i$ of the strings $\sigma_i$,
in their order. 
If the reduced form of $\sigma$   has the form $b^na^m$, with $ n,m < \omega$,  let
$\prod _{i \in I}  \sigma _i  = b^na^m$; otherwise, let 
$\prod _{i \in I}  \sigma _i  = \Omega $.
The above definition includes the case $\prod_ \emptyset = e
=b^0a^0$. 

Again, it is intuitive that $S$ is
a $\leq \omega $-semigroup, but a full proof
seems to need a few tricky details. 
See Appendix I for a full proof.

\begin{remark} \labbel{diffic}   
Before explicitly treating  the general case of a complete
semigroup in Theorem \ref{abe},
we briefly discuss the difficulties which arise.
The general case is more involved than the arguments in 
the previous example, since if in some
infinitary semigroup $S$ 
we naively set some $\mathbb Z$-indexed
product  $ \dotsc aaaaabbbbb \dotsc $ to be equal to $  e$,
then by (N)  both  $ \dotsc aaaaa $
and $ bbbbb \dotsc $ should be defined in $S$ and not absorbing.
We expect that this situation is admissible
(though we have not checked the details yet), but then, going on,
we have that $ bbbbb \dotsc  \ b$,
$ bbbbb \dotsc  \ bb$ and
$ bbbbb \dotsc  \ bbbbb \dotsc$ should be defined as well
and, iterating transfinitely the construction, we possibly get a semigroup
with domain as large as a proper class. 
Hence, if we want  $ \vert  S  \vert   = \omega $,
it is surely simpler to assume that 
 $ \dotsc aaaaabbbbb \dotsc \neq e$.
The issue is solved by dealing with  some ``(possibly partial) reductions'', 
instead of   dealing with a (fully) reduced form.
We will present full details in Definitions \ref{full} - \ref{fullred}. 

In any case, an $\mathbb N$-indexed product of the form
$aaaaa\dots$ cannot have a right inverse,
unless $a=e$; see \cite[Proposition 6.7]{smgrnc}. 
\end{remark}

\section{Pseudonull sequences} \labbel{pnsec}   

We are now going to introduce the class of pseudonull
sequences, which are a broad generalization of
those finite strings which reduce to $e$
in a finitary semigroup, under the assumption that $ab=e$.
The ``right'' definition is far from being 
straightforward, due to the difficulties hinted in     
Remark \ref{diffic}.

\begin{definition} \labbel{full}    
For arbitrary linearly ordered sets, we consider
$I$-indexed sequences
whose elements, again,  are only  $a $ and $ b$.

The class of \emph{pseudonull} sequences is the smallest class containing
  \begin{enumerate}[(i)]        
\item 
the empty sequence, and
\item
any sequence $( s_i) _{i  \in I} $ such that 
$I$ can be partitioned into convex subsets
$I_j$ ($j \in J$) in such a way that, for every $j \in J$,
the subsequence  $( s_i) _{i  \in I_j} $ is pseudonull, and
\item
any sequence $( s_i) _{i  \in I} $ such that $I$ 
has a minimum $i_1$, a maximum $i_2$, 
$s _{i_1}=a $, $s _{i_2}=b $ and
$( s_i) _{i  \in I \smallsetminus \{ \,  i_1, i_2 \, \} } $ is 
a pseudonull sequence. 
  \end{enumerate}    

See Definition \ref{full2} below for a more formal definition.
 \end{definition}   

In particular, by (i) and (iii),
if $I= \{ \,  i_1, i_2 \, \} $ and $s _{i_1}=a $,
$s _{i_2}=b $, then the
sequence $( s_i) _{i  \in I} $ is pseudonull.
In what follows, 
we will denote the above sequence by $ab$,
namely we will consider the sequence as a string and we will
work up to reindexing.  We will use a similar convention
for countable sequences. In fact,
 linearly ordered sequences 
are a generalization of strings and some authors  plainly call
them strings. In this and the next section,
juxtaposition always denotes string concatenation.

For example, the sequences $ababab\dotsc$  ($\mathbb N$-ordered) and
$\dotsc ababab\dotsc$ ($\mathbb Z$-ordered),
hence also $ (aaababab \allowbreak \dotsc)bb$ and  
$aa(\dotsc ababab \dotsc)bb$ are pseudonull.
 On the other hand, $\dotsc aaaaabbbbb \dotsc$,
is not pseudonull, essentially since in (ii)  we need to deal with
a union of \emph{disjoint} convex suborders. 
See below for more explicit details.
By the way, note that $\dotsc aaaaabbbbb \allowbreak \dotsc$
is a ``union''   of (not disjoint) pseudonull sequences.

In order to carry over appropriate
transfinite inductions, we need an evaluation of the ``complexity''
of pseudonull sequences. By the way,
we get a more formal definition of a pseudonull sequence.

\begin{definition} \labbel{full2}
We  define the class of \emph{pseudonull sequences
of degree $ \leq \alpha$}, for $\alpha$ an ordinal,
by transfinite induction. The empty sequence,
 given by (i) in Definition \ref{full},
is of degree $ \leq 0$. 
 A sequence given by (ii)
is pseudonull of degree $\leq \alpha$,
where $\alpha$ is the smallest ordinal strictly larger than all  the degrees
of the involved subsequences.
A sequence given by (iii)
is pseudonull of degree $\leq \alpha +1$, provided  
$( s_i) _{i  \in I \setminus \{ \,  i_1, i_2 \, \} } $ 
is pseudonull of degree $\leq \alpha $.

In particular, a sequence $\sigma$  is pseudonull if and only if
 there is some ordinal
$\alpha$ such that $\sigma$  is pseudonull of degree $\leq \alpha $,
hence there is the least ordinal $\bar \alpha$, the \emph{degree of}
$\sigma$,  such that $\sigma$ 
is  pseudonull of degree $\leq \bar \alpha $. 
 \end{definition}   

With the above more formal definition, we now give the exact details for 
the comment  asserting that   the sequence 
$\dotsc aaaaabbbbb \dotsc$
is not pseudonull.
By construction, (i) and (iii) cannot witness that the sequence 
$\dotsc aaaaabbbbb \dotsc$
is pseudonull, so that, were 
$\dotsc aaaaabbbbb \dotsc$  pseudonull, this should be
witnessed by (ii). But this is clearly impossible, since it would
entail that some sequence whose only elements are $a$, or $b$,
is pseudonull, which is not the case, since each nonempty pseudonull sequence
must contain at least both an occurrence of $a$ and an occurrence of $b$
(formally, this is proved by
induction on $\alpha$; compare the proof of item (a) 
in Lemma \ref{cl}  below).  

\begin{definition} \labbel{fullred}
If  $ \sigma = ( s_i) _{i \in I} $ is a sequence, let us say that 
some (possibly empty)
sequence $ \tau  = ( s_r) _{r \in R} $ is \emph{a reduction of} 
$\sigma$, or \emph{a reduced form of} $\sigma$,
 if $R \subseteq I$ and there is a \emph{reducing family}
 $( K_\ell) _{ \ell \in L} $
of pairwise disjoint convex subsets of $I$ such that, for each $\ell \in L$,   
the sequence $( s_i) _{i \in K_\ell}$ is pseudonull,
and, moreover, $R= I \setminus \bigcup_{ \ell \in L}  K_\ell $.
Note that we do not assume that $R$ is convex, though
each $ K_\ell$ is required to be convex. 
For short, $\tau$ is a reduction of $\sigma$ if $\tau$
is obtained from 
$\sigma$ by removing a family of pairwise disjoint pseudonull 
convex
subsequences. As above, the sequence 
$\dotsc aaaaabbbbb \dotsc$ witnesses that
(contrary to the example presented in Section \ref{XyX})
 we do not always necessarily 
have a canonical ``total''  reduction.  

A sequence $\sigma$ is \emph{regular}
if it has a \emph{(regularizing) reduction} $\tau$ of the form  
  $b^na^m $, with $n,m \geq 0$.
In the  case of a regular sequence, the reduction 
providing   $b^na^m $ is ``total''
(this is intuitively clear, but a formal proof is quite tricky, see
items (f) and (h) in Lemma \ref{cl}  below).

Note that in the above definition of regular
we include the  possibility that $n=m=0$,
that is, $\tau$ can be the empty sequence.
In particular,
a pseudonull sequence is regular. On the other hand,
by considering clause (ii) in Definition \ref{full}, a 
 sequence with an empty  reduction  is pseudonull.
Thus a 
sequence is pseudonull 
if and only if it has a  reduction of the form $n=m=0$.
 \end{definition}   

For notational simplicity, it is convenient to refer
to pseudonull, regular, etc., indexes, rather than sequences.

\begin{definition} \labbel{fullsubs}
If some fixed sequence $ ( s_i) _{i \in I} $ is intended 
 and $H \subseteq I$,
we will say that $H$ is \emph{pseudonull} if 
the subsequence $ ( s_i) _{i \in H} $ is pseudonull and
we will write ``the degree of $H$'' 
for  ``the degree of $ ( s_i) _{i \in H} $''. Similarly, 
we say that $H$ 
is \emph{regular} if   $ ( s_i) _{i \in H} $  is regular. If this is the case,
we will also say that $H$ \emph{has a reduction
of the form} (or that $H$ \emph{reduces} to)  
  $b^na^m $ if  $ ( s_i) _{i \in H} $ has a reduction of the form
 $b^na^m $. 
\end{definition}

\section{The main lemma} \labbel{mlem} 

\begin{lemma}\labbel{cl} 
Let us fix some sequence $ \sigma = ( s_i) _{i \in I} $
in which only the elements $a$ and $b$ occur.
Recall the convention in Definition \ref{fullsubs}.

(a)  Suppose that $ \sigma $ is  pseudonull,
 $P$ is an initial suborder of $I$ 
and $Q$ is the complementary final suborder.
Then there is $n \geq 0 $
such that  
$ P$ has a reduction of the form $a^n$ 
and $Q$ has a reduction of the form $b^n$.
 
(a1) In particular, if $\sigma$ is pseudonull and $I$ has a minimal
 element $i_1$ (a maximal element $i_2$), then 
$s_{i_1}=a$ (respectively, $s_{i_2}=b$).

(b)   If $\sigma$ is regular, then there are a finite reducing family
$( K_\ell) _{ \ell \leq r+1} $   and finite convex subsets 
$( R_\ell) _{ \ell \leq r} $ of $I$ such that 
$I$ is the disjoint union of $K_0$, $R_0$, $K_1$, $R_1$, \dots,
$K_{r+1}$, ordered in that way,  $R_0$ is the set of indexes
for a sequence of the form $b^{n_0}$,
$R_1$ is the set of indexes
for a sequence of the form $b^{n_1}$, \dots,
until we get some $R_s$ with an associated sequence  of the form   
$b^{n_s}a^{m_s}$, with all the ensuing 
$ R_\ell$ having associated sequences of the form
$a^{m_\ell}$. (We have described the most general form:
some ingredients might be missing, the ordered set  
$I$ might start directly with  $R_0$, that is, 
$K_0$ might be empty; moreover, 
the $a^{m_\ell}s$ might be missing, when the reduction has the form $b^n$ or,
conversely,  the $b^{m_\ell}s$ might be missing, when
 the reduction has the form $a^m$; in the extreme case 
when $\sigma$ is pseudonull, the above representation reduces
to a one-element reducing family consisting of $I$ alone.)

(c)  If $ \sigma = ( s_i) _{i \in I} $ is  regular 
and $H$ is a convex subset of $I$, then $H$ is regular, namely,
$ ( s_i) _{i \in H} $ is a regular  sequence.

(d) If $H, K \subseteq I$ are both pseudonull and convex,
then   $H \cap K$ is pseudonull (and, obviously, convex, possibly,
empty).

(e) Under the same assumptions as in (d), both 
$H \setminus K$ and $K \setminus H$ are pseudonull. 

(f) If $ \sigma  $ is a regular  sequence,
then the $n$ and $m$ in any reduction  
$b^na^m$ of $\sigma$ are uniquely determined.

(g1) If $\sigma$ is pseudonull, then every reduction of $\sigma$
is pseudonull.

(g2) If $\sigma$ has some pseudonull reduction, then $\sigma$ is
pseudonull.

(h) If $\sigma^*$
is a (not necessarily regularizing) reduction of $\sigma$,
then $\sigma$ is regular if and only if   $\sigma^*$ is regular and, 
if this is the case, the regular
reductions of $\sigma$ and $\sigma^*$ (unique by (f))
 are equal.
\end{lemma} 

\begin{proof} 
(a) The proof  is by transfinite induction on
the degree of  $\sigma$.
The conclusion is straightforward 
if $\sigma$ is the empty sequence, as given by (i)
in Definition \ref{full}.

Let $\sigma$ be given by (ii) in Definition \ref{full}, so that 
$I$ is the union of pairwise disjoint 
pseudonull convex subsets $I_j$ ($j \in J$) of smaller degree,
according to Definition \ref{full2}.
If both $P$ and $Q$ are unions of 
a subfamily of the $I_j$,
then the conclusion holds with $n=0$
(this case includes the straightforward situation in 
which either $P$ or $Q$ is empty, thus, correspondingly,
$Q$ or $P$ equals $I$).
Otherwise, since $P$ and $Q$ partition $I$,
$P$ is an initial segment of $I$,
$\sigma$ is given by (ii)
and the $I_j$ partition $I$ into convex subsets,
 there is 
$\bar{\jmath} \in J$ such that both 
$P \cap I _{\bar{\jmath}} $ and
 $Q \cap I _{\bar{\jmath}} $ are nonempty and 
$P \cap I _{\bar{\jmath}} $ is an initial segment
of $ I _{\bar{\jmath}} $.
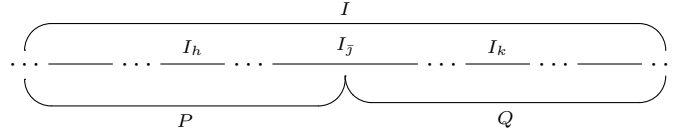
\begin{equation*} 
\xymatrix{
\cdots  \ar@{-}`u[r] `[rrrrrr]^{I} [rrrrrr]
 \ar@{-} 
`d[rr] `[rrr]_{P} [rrr]
 \ar@{-}[r]  &
  \cdots  \ar@{-}[r]^{I_h}  &
\cdots   \ar@{-}[rr] ^{I_{\bar \jmath}} &
 \ar@{-} `d[rr] `[rrr]_{Q} [rrr] &
  \cdots  \ar@{-}[r]^{I_k}  &
\cdots   \ar@{-}[r]  & \cdots 
 }
 \end{equation*}     
If this is the case, then
the degree of $ I_{\bar \jmath} $
is $<$ than the degree of $I$.
Hence we can apply  the inductive hypothesis,  
thus there is some $n \geq 0$ such that 
$P \cap I _{\bar{\jmath}} $  
has a reduction of the form $a^n$ and  
 $Q \cap I _{\bar{\jmath}} $
has a reduction of the form $b^n$.
Because of the representation of
$I$ given by (ii), $P \setminus  I _{\bar{\jmath}} $
is either empty, or the disjoint union 
of pseudonull convex sets.
By putting together such sets with the above reduction
of $P \cap I _{\bar{\jmath}} $  , we get a reduction of the form $a^n$
for $P$. Similarly,
$Q$  has a reduction $b^n$. 

Finally, let $\sigma$ be given by (iii) in Definition \ref{full},
thus  $ I^*=  I \setminus \{ \,  i_1, i_2 \, \}  $
is  pseudonull with smaller degree. As in the previous paragraph,
the cases when either $P$ or $Q$ is empty 
is straightforward. If 
 $\vert P \vert = 1$, namely, $P = \{ \,  i_1\, \}$,
since $P$ is initial, then $ Q= I \setminus \{ \,  i_1 \, \} 
=  I^* \cup  \{ \,  i_2 \, \}   $
thus $Q$  has reduction $b$, since 
$I^*$ is pseudonull and $s _{i_2 } =b $,  namely, we get $n=1$.  
 The case when $Q$  has just one element
  is symmetrical.
Otherwise, both $P$ and $Q$ have nonempty intersection with
$I^* $. By the inductive assumption,
and since $I^*$ is pseudonull,
 there is some $n$ such that $P \setminus \{ \, i_1 \, \}$
reduces to $a^n$ and   
$Q \setminus \{ \, i_2 \, \}$
reduces to $b^n$.
Since $s_{i_1}=a$ is the first element of $\sigma$,
the same reduction as above furnishes a reduction  
of the form $a^{n+1}$ for $P$ and,
symmetrically, a reduction  
of the form $b^{n+1}$ for $Q$.
Note for later use that if $\sigma$ is given by (iii),
then $n \geq 1$. 

(a1) is immediate from (a) by taking 
$P = \{ \, i_1 \, \}$, $P= I \setminus \{ \, i_2 \, \}$,
respectively.  

Item (b) is almost immediate using
clause (ii) in the definition of a pseudonull sequence
(and (a) above is not needed).
Indeed, given an arbitrary 
reducing family  $( K^*_\ell) _{ \ell \in L^*} $,
simply group together each set of $ K^*_\ell$s
which are not intertwined with some index 
giving some $a$ or some $b$ in  the reduction.
Since by definition the $K^*_\ell$
are convex, pseudonull and pairwise disjoint, 
 by clause (ii) in the definition of 
a pseudonull sequence, the union of each of the above groups
is still pseudonull.
This provides the desired representation.

(c) We first consider the case when \emph{(c1) $\sigma$ is pseudonull}.
We may also assume that $H$ is neither an initial segment
of $I$, nor a final segment of $I$,  since such cases are covered
by (a). The proof is by induction on the degree of $\sigma$.
As usual, the base case when $\sigma$ is empty is straightforward.

Let $\sigma$ be given by (ii) in Definition \ref{full},
as witnessed by subsets $I_j$ ($j \in J$), which,
 according to  Definition \ref{full2}, have degree strictly less
than the degree of $\sigma$. If $H$ is contained in some
$I_j$, we can thus apply the inductive assumption in order to
get the result. As in the proof of (a), if $H$ is a union
of some members of the family  $I_j$ ($j \in J$),
then $H$ is pseudonull, in particular, $H$ is regular.
Compare Definition \ref{fullred}. 
Otherwise, since $H$ is convex by assumption,
$H$ intersects (and does not contain)
one or two $I_j$, say, $I_{j_1}$ on the left and, 
possibly, $I_{j_2}$ on the right.
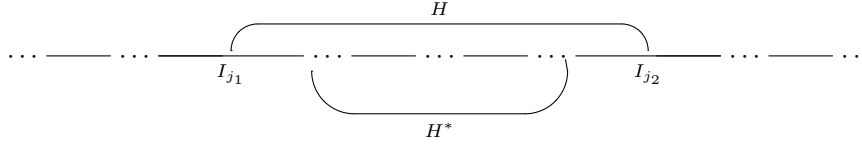
\begin{gather*} 
\xymatrix{
\cdots 
 \ar@{-}[r]  &
  \cdots  \ar@{-}[r]   \ar@{-}[rr]_{I_{j_1}}  &
 \ar@{-}`u[r] `[rrrr]^{H} [rrrr] & 
\cdots   \ar@{-}[r]  \ar@{-}@<-6pt>`d[rr] `[rr]_{H^*} [rr] &
  \cdots  \ar@{-}[r]  & \cdots  \ar@{-}[rr] _{I_{j_2}}  
& \ar@{-}[r] &
\cdots   \ar@{-}[r]  & \cdots &
 }
 \end{gather*}     
Since $H$ and $I_{j_1}$ are convex, $H \cap I_{j_1}$ 
is a final subset of $I_{j_1}$, hence by (a) it is regular 
and has a reduction of the form $b^n$.
If $H$ does intersect some $I_{j_2}$ on the right,
then $H \cap I_{j_2}$ is an initial segment of $ I_{j_2}$, hence,
 again by (a),  $H \cap I_{j_2}$ is regular
with a reduction of the form $a^m$.
The subset $H$ might also possibly contain 
a set of $I_j$; this means that
the convex set  $H^*= H \setminus ( I_{j_1} \cup I_{j_2})$
is a union of a family pairwise disjoint pseudonull convex subsets of $I$,
thus, by \ref{full}(ii),  $H^*$ is pseudonull.
If we join  $H^*$ with the reductions for 
$H \cap I_{j_1}$ and for $H \cap I_{j_2}$
(note that the above three sets are pairwise disjoint),
we get a reduction of $H$ of the form  $b^na^m$,
thus $H$ is regular.
 
If $\sigma$ is given by (iii) in Definition \ref{full},
the conclusion is immediate from the inductive hypothesis,
since, as mentioned at the beginning of this proof
of (c1), we may assume that $H$ 
is neither an initial  nor a final segment of $I$, thus $H$ 
is contained in $I \setminus \{ \, i_1, i_2\, \} $,
which is pseudonull of degree smaller than $\sigma$. 

Having proved the special case (c1), let us prove (c)
in general. 
Since $\sigma$ is regular, we may assume to have a
reducing family as in (b). The case when $H$ is a subset of some
$K_ \ell$ is covered by (c1) above, since the
$K_ \ell$ are pseudonull. Otherwise
$H$, being convex,  possibly contains a certain number of  $K_ \ell$
(possibly, none), 
a certain number of
$R_ \ell$ (possibly, none) and possibly intersects at most two of them
without containing them. The argument now is similar to case (c1)(ii) above,
except that some $R_ \ell$ does contribute to the 
reduced form. 

In detail, if the initial segment of $H$ intersects some 
$K_ {\bar \ell}$, we get from (a) that $H \cap K_ {\bar  \ell}$
has a reduction of the form $b^{n}$, since it is a final segment
of the pseudonull set $K_ {\bar  \ell}$. Then $H$ contains $R_ {\bar  \ell}$,
which contributes, say, $b^{n_ {\bar \ell}}$ to the reduced form of
$H$. Then $K_ {\bar  \ell+1}$ is pseudonull, 
$R_ {\bar  \ell+1}$ contributes $b^{n_ {\bar  \ell+1}}$
and so on, until some $K_ \ell$ contributes
$b^{n_ { \ell}}a^{m_ { \ell}}$. Then, symmetrically, we have 
finite sequences of $a$ as contributions. 
Possibly, the final segment of $H$ intersects some
$K_ {\bar {\bar \ell}}$; in this case we get from (a)
 that $H \cap K_ {\bar {\bar \ell}}$
has a reduction of the form $a^{m}$, being an initial segment
of the pseudonull set $K_ {\bar {\bar \ell}}$. Thus we have
first a reduction of $H \cap K_ {\bar  \ell}$; then the above description
constitutes a reduction of $H \setminus ( K_ {\bar  \ell} \cup  K_ {\bar {\bar \ell}})$
and finally we have a reduction of $H \cap K_ {\bar {\bar \ell}}$ on the
right.
If we combine the above reductions, we get a regular reduction of
$H$, since the description in (b) involves only finitely many
subsets, hence the reduction is a finite sequence.

Of course, as in the statement of (b), 
we have described the most general situation;
some of the above contributions might be missing. Also,
the initial/final parts of $H$ might intersect some $R_ \ell$;
in any case, we get contributions of the form $b^{n}$ or $a^{m}$.
The important aspect of the argument is that, in any case,
each occurrence of $b$, if any, comes before any occurrence
of $a$, if any.  

We now turn to item (d). The case is straightforward if either $H$ 
is contained in $K$ or conversely. In particular,
(d) holds if either $H$ or $K$ are empty.
Also the case when $H \cap K = \emptyset $ is straightforward. 
The general case is proved by induction on the natural sum of
the degrees of $H$ and $K$. Recall that the natural sum
 is a strictly monotone operation  on both arguments \cite{bach,sier},
a property which makes it suitable for the induction.

As we mentioned, the case when  $H$ or $K$ is empty
is straightforward. In particular, this includes the base case
of the induction when both $H$ and $K$ are empty.

Suppose that $H$ is pseudonull as given by \ref{full}(ii), namely,
$H$ is the union of pseudonull pairwise disjoint 
convex subsets $I_j$ ($j \in J_H$). 
If $H \cap K$
is a union of a subfamily of such pseudonull disjoint subsets of $H$,
then (ii) applies to $H \cap K$, as well, thus
$H \cap K$ is pseudonull.
Otherwise, since  $H$ 
and $K$ are convex and  $K$ is not contained in $H$,
there is some pseudonull subset $I_{\bar \jmath}$
of $H$ such that  $K \cap I_{\bar \jmath} \neq \emptyset $
but $K$ does not contain $I_{\bar \jmath}$.
By the inductive hypothesis, $K \cap I_{\bar \jmath}$ is pseudonull.
But then  $H \cap K$ is pseudonull, again by applying (ii),
being the disjoint union
of $K \cap I_{\bar \jmath}$ with the (possibly empty) family
of those pseudonull subsets $I_j$ of $H$ which lie
on the right or on the left of $I_{\bar \jmath}$
(according to whether or not there is some $h \in H$ 
such that $h < k$, for every $k \in K$).

The case in which $K$ is pseudonull as given by \ref{full}(ii)
is symmetrical, hence we can suppose that 
both $H$ and $K$ are pseudonull as given by (iii).
Then $H \cap K$ has a minimal element $i_1$
and a maximal element $i_2$ with $s _{i_1}=a $ and
$s _{i_2}=b $, since, without loss of generality, we may
assume that $H \cap K \neq \emptyset  $.
Indeed, suppose that, say, some $h \in H$ is 
$<$ than all the elements of $K$. Then the minimal 
element of $H \cap K$ is the minimal element of $K$ 
and the maximal element of   $H \cap K$ is the maximal element of $H$.
 Now, since both $H$ and $K$ are pseudonull as witnessed by (iii),
if we remove their maximal and minimal elements,
we still get pseudonull convex sets $H'$ and $K'$.
By the inductive hypothesis,
$H' \cap K'$ is pseudonull. 
But then $H \cap K$ is obtained from $H' \cap K'$ by adding  $i_1$
at the beginning and  $i_2$ at the end, so that 
$H \cap K$ is pseudonull by (iii), since 
$s _{i_1}=a $ and
$s _{i_2}=b $.

(e) First, assume that \emph{(e1) $K \subseteq H$
and, moreover, $K$ is either an initial segment
or a final segment of $H$}.
As remarked at the end of Definition \ref{fullred}, $ K$ being pseudonull
 means exactly that $K$  has an empty sequence as a
reduction. Since $H \setminus K$ is the complement 
of $ K$ in $H$, 
$H \setminus K $
has an empty sequence as a
reduction  by (a), and, as above,  this means
that $H \setminus K$ is pseudonull. 

Now consider \emph{(e2) the case when $H$ and $K$ 
are $ \subseteq $-incomparable}.  The case when 
$H \cap K = \emptyset $ is straightforward. Otherwise, 
 $H \setminus K = H \setminus (H \cap K)$
and $H \cap K$ pseudonull, by (d).  
Moreover,  $H \cap K$ is either an initial segment
or a final segment of $H$, since $H$ and $K$ are convex.
By (e1) above, 
$H \setminus K = H \setminus (H \cap K)$ is pseudonull.

Up to symmetry, it remains only to treat the case when \emph{(e3) 
$K \subseteq H$
but $K$ is neither an initial nor a final segment of $H$},
namely, $K$ lies ``strictly in the middle''.
We will prove this case by induction on the degree of $H$.
Note that in this case $H \setminus K$ is not convex, unless
$K$ is empty. 

As usual, the base case when $H$ is empty is straightforward.
Next, assume that  $H$ is given by (ii) in Definition \ref{full}, so
$H$ is the union of pseudonull pairwise disjoint 
convex subsets $I_j$ ($j \in J_H$). 
Then $K$, being convex, contains 
a (possibly empty) subfamily of the $I_j$, say, 
$K$ contains each $I_j$ for $j \in J^*$, with $J^* \subseteq J_H$.
Moreover, $K$ may possibly intersect  some
$I_{j_1}$ and some $I_{j_2}$, without containing them.
Then $H \setminus K$ is the union of the $I_j$ 
($j \in J_H \setminus J^*$, $j \neq j_1$, $j \neq j_2$)
plus possibly $I_{j_1} \setminus K$ and $I_{j_2} \setminus K$.
By the assumptions in (ii), each $I_j$ is pseudonull. 
By the already proved case (e2), 
if $K \not \subseteq I_{j_1}$ and  $K \not \subseteq I_{j_2}$, then 
both 
$I_{j_1} \setminus K$ and $I_{j_2} \setminus K$
are pseudonull, as well, hence (ii) applies 
to $H \setminus K$, which therefore is pseudonull. 
So far, we have not used the inductive hypothesis.

An ``exceptional case'' remains to be treated: when there is some
$ \bar{\jmath}  \in J_H$ such that $K  \subseteq I_{ \bar{\jmath} }$.
If such a $ \bar{\jmath} $ exists, it is unique,
since the $I_j$ are pairwise disjoint. 
In this case $I_{ \bar{\jmath} }$ has degree
less than the degree of $H$, thus by the inductive assumption
(with $I_{ \bar{\jmath} }$ in place of $H$) we get that
$I_{ \bar{\jmath} } \setminus K$ is  pseudonull.
Now $H \setminus K$ is pseudonull by (ii),
being the union of $I_{ \bar{\jmath} } \setminus K$ 
plus the remaining $I_j$s ($j \neq \bar{\jmath}$).
Indeed, the sets in the above family are  pairwise disjoint and, within  
$H \setminus K$, are convex.
   
If $H$ is given by \ref{full}(iii), then $K$ is contained 
in $H^* = H \setminus \{ i_1,  i_2 \} $, since in the
present case (e3) we are assuming that $K$ lies strictly in the middle of $H$.
By the assumptions in (iii) $H^*$ is pseudonull and, 
by the inductive definition of the degree, $H^*$
has degree strictly less than the degree of $H$. 
By the inductive hypothesis, $H^* \setminus K$
is pseudonull. But
$H \setminus K$ is obtained from 
$H^* \setminus K$ by adding one occurrence
of $a$ to the left and an occurrence of $b$
to the right, hence (iii) applies witnessing that 
$H \setminus K$  is pseudonull.

(f) Suppose that $\sigma$ has a reduction of the form
$b^na^m$, namely,
there is a family $( K_\ell) _{ \ell \in L} $
of pairwise disjoint convex pseudonull subsets of $I$ 
such that  
$I \setminus \bigcup  _{ \ell \in L}  K_\ell$
 is finite and is the set of indexes for a sequence
of type  $b^na^m$. Without loss of generality, we may assume that
 the family $( K_\ell) _{ \ell \in L} $ is of the form
given in (b) (note that the construction in the proof of (b) does not
modify $ \bigcup  _{ \ell \in L}  K_\ell $).
We will show that 

\emph{(f*) for any other (not necessarily regularizing)
 reduction $(I_j) _{ j \in J} $  of $\sigma$,
$\bigcup  _{ j \in J}  I_j \subseteq \bigcup  _{ \ell \in L}  K_\ell $,
in words, those indexes providing the realization
of  $b^na^m$ do not belong to the $I_j$.}  

Repeating the argument with  the two reductions
exchanged, we get that, if both reductions are regularizing,
 the two reductions are equal.
Actually, we get slightly more: those indexes giving
$b^na^m$ are the same.

So let us prove (f*). Suppose by contradiction that $(I_j) _{ j \in J} $ 
provides another reduction such that
some occurrence of $b$ or  $a$ 
with indexes not in $ \bigcup  _{ \ell \in L}  K_\ell $
has index in some $I_{ \bar{\jmath} }$.
Recall that we have assumed that 
$( K_\ell) _{ \ell \in L} $ has the form
given in (b). In the most involved case,
$I_{ \bar{\jmath} }$ intersects
some $K_t$ without intersecting the preceding
$R_{t-1}$. Then  $I_{ \bar{\jmath} }$, being convex, 
contains $R_{t}$, contains $K_{t+1}$, \dots, until
$I_{ \bar{\jmath} }$ intersects some 
$K_\ell$ without intersecting the following
$R_{\ell}$.
Applying (e) two times, 
$I^* =I_{ \bar{\jmath} } \setminus (K_{t} \cup K_{\ell})$
is pseudonull. Note that $K_{t}$
 lies ``on the left'' of   $I_{ \bar{\jmath} }$, hence
$I_{ \bar{\jmath} } \setminus K_{t}$ is convex.
Arguing in the same way ``on the right'',
we get that $I^*$ is convex.   
Now, $I^*$ contains the finite $R_{t}$ ``just at the beginning'', hence
$I^*$ has a minimal element $i_{1}$. 
In the most general case, $( s_i) _{i  \in  R_{t}}$ 
has the form $b^{m_t}$, with $m_t > 0$,  in particular, 
$s _{i_{1}} =b $, but this contradicts (a1),
since we have proved  that $I^*$ is pseudonull.
Thus  $( s_i) _{i  \in  R_{t}}$ 
is $a^{n_t}$, but then, due to the general
form given in (b), also $R_{\ell-1}$ has the form
$a^{n_{\ell-1}}$, thus, if  $i_{2}$ is the maximal
element of $I^*$, $s _{i_{2}} =a $, 
and  this again contradicts (a1).

So far, we have considered the most involved case
in which $I_{ \bar{\jmath} }$ intersects
some $K_t$ ``at the beginning''. 
On the other hand, it might happen 
that the first elements of $I_{ \bar{\jmath} }$ 
come from some $R_t$. If this is the case, just argue as above
by considering   $I^* =I_{ \bar{\jmath} } \setminus K_{\ell}$
or, if also the last elements of 
$I_{ \bar{\jmath} }$ 
come from some $R_ \ell$,
simply consider $I^* =I$. In any case, by assumption,
at least one occurrence of $b$ or  $a$ 
with indexes not in $ \bigcup  _{ \ell \in L}  K_\ell $
has index in  $I^*$.
By construction, and the general form given by (b),
if we get more than one $a$ or $b$,
we have that all the $a$ follow all the $b$.
In any case,  $I^*$ has a minimal index carrying
a $b$, or a  maximal index carrying
an $a$, in each case contradicting (a1).
 We have got a contradiction in each case,
so (f*), hence also (f) are proved.

(g1) is immediate from (f)
and the last sentence in Definition \ref{fullred}. 

(g2) Suppose that $\sigma^*$ is a pseudonull
reduction of $\sigma$, as given by the reducing family
$( K_\ell) _{ \ell \in L} $, with each 
 $K_\ell \subseteq I$, so that the index set of 
$\sigma^*$ is $I^* =
I \setminus \bigcup  _{ \ell \in L}  K_\ell$. 

 We will prove
(g2) by induction on the degree of $\sigma^*$.
If $\sigma^*$ is empty, then 
$I = \bigcup  _{ \ell \in L}  K_\ell$, so that 
 $\sigma$ is pseudonull 
by (ii) in Definition \ref{full}, as  witnessed by the family
$( K_\ell) _{ \ell \in L} $ itself.

Suppose that $\sigma^*$ is pseudonull as witnessed by 
some family $(I_j) _{ j \in J} $ 
as in (ii) in Definition \ref{full}.
For every $j \in J$, let $L_j$
be the set of those $\ell$ such that there are 
$i < i' \in I_j$ with  $i <  K_\ell  <  i' $, namely,
$i <  k <  i' $ (in $I$), for every $k \in K$.
Then, for every $j \in J$, 
 $( K_\ell) _{ \ell \in L_j} $ is a reduction
of  $ I_j^\diamondsuit =  I_j \cup  \bigcup  _{ \ell \in L_j}  K_\ell$,
whose reduced set of indexes is $I_j$.
By the assumptions, $I_j$ is pseudonull of degree 
less than the degree of $I$, so that, by the inductive hypothesis,
$I_j^\diamondsuit $ is pseudonull (note that, by construction,
$I_j^\diamondsuit $ is convex as a subset of $I$).    
It is not necessarily the case that the union of 
the $I_j^\diamondsuit $s is the whole of $I$,
since there might be certain $\ell$s which belong to no 
$L_j$. However, it is enough to add the corresponding 
$ K_\ell$s in order to get a representation of $I$  
as in (ii) in Definition \ref{full}.
In more detail, if $L^\circ =  L \setminus 
\bigcup  _{ j \in J}  L_j$, then
the family consisting of the  $I_j^\diamondsuit $ ($j \in J$)
together with the $ K_\ell$ ($\ell \in L^\circ$)
witnesses that $I$ (that is, $\sigma$) is pseudonull by (ii).

Finally, suppose  that $\sigma^*$ is pseudonull as witnessed by 
(iii) in Definition \ref{full}, thus $I^*$ has a minimal element
$i_1$, a maximal  element
$i_2$ and $I^{\bullet}= I^* \setminus \{ \, i_1, i_2 \, \} $ is pseudonull.
Let $L^* = \{ \, \ell \in L \mid  i_1 <  K_\ell < i_2 \, \} $.
The family  $( K_\ell) _{ \ell \in L^*} $ is a reduction
of  $ I^\diamondsuit =  I^{\bullet} \cup  \bigcup  _{ \ell \in L^*}  K_\ell$,
whose reduced set of indexes is $I^{\bullet}$. By the
inductive assumption, $I^\diamondsuit$ is pseudonull. 
Then, by item (iii) in Definition \ref{full}, 
$I^{\heartsuit}=
I^\diamondsuit \cup \{ \, i_1, i_2 \, \}$ is pseudonull.
Now note that, since each $ K_\ell$ is convex, 
if $ \ell \notin L^*$, then in $I$ either  $ K_\ell < i_1$,
or $ K_\ell > i_2$.
This implies that $I=I^{\heartsuit}
\cup \bigcup  _{ \ell \notin L^*}  K_\ell$,
which is pseudonull by (ii) in Definition \ref{full}.

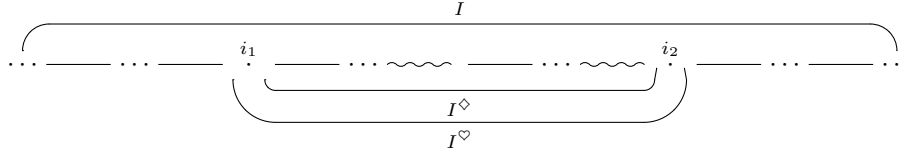
\begin{gather*} 
\xymatrix{
\cdots  \ar@{-}`u[r] `[rrrrrrrr]^{I} [rrrrrrrr]
 \ar@{-}[r]  &
  \cdots  \ar@{-}[r]  &
 \ar@{-}@<-6pt> `d[r] `[rrrr]_{I^{\heartsuit}} [rrrr]
 \ar@{-} @<+6pt>`d[r] `[rrrr]_{I^{\diamondsuit}} [rrrr]
 \ \overset{ i_1}{\cdot}\  
\ar@{-}[r]  & \cdots  \ar@{~}[r] & \ar@{-}[r] & \cdots \ar@{~}[r] &
 \ \overset{ i_2}{\cdot}\ 
 \ar@{-}[r]  & \cdots  \ar@{-}[r]  & \cdots}
\\
\textit{the $K_\ell$ represented by solid lines, the curled lines representing elements
of $I^{\bullet}$} 
 \end{gather*}

(h) Suppose that the reducing family giving $\sigma^*$ from $\sigma$ 
is $(I_j) _{ j \in J} $.

First assume that $\sigma$ is regular, with regularizing family
$( K_\ell) _{ \ell \in L} $, giving the reduced sequence
$\sigma^r$.  For $j \in J$ and $\ell \in L$, let  
$I_{j,\ell} = I_j \cap K_\ell $.
By (d)
 each $I_{j,\ell}$ is pseudonull.
By (f*) proved above,
$\bigcup  _{ j \in J}  I_j \subseteq \bigcup  _{ \ell \in L}  K_\ell $,
so that the family $(I_{j,\ell}) _{ j \in J,\ell \in L} $
is still a reducing family giving $\sigma^*$.
By (g), for every $\ell \in L$, the set 
 $ K^\diamondsuit  _\ell  =K_\ell \setminus \bigcup  _{ j \in J}  I_{j,\ell}$ 
is pseudonull, so that  $(K^\diamondsuit  _\ell) _{\ell \in L}  $
is a reducing family for $\sigma^*$.
The index set of   $\sigma^*$ is $I \setminus \bigcup  _{ j \in J}  I_j$,
 the index set of  the regular reduction $\sigma^r$ of $\sigma$ 
 is $I \setminus \bigcup  _{ \ell \in L}  K_\ell$,
 so that  $(K^\diamondsuit  _\ell) _{\ell \in L}  $
reduces $\sigma^*$ exactly to $\sigma^r$.
 
Conversely, suppose that $\sigma^*$ is regular, with regularizing family
$( K^*_\ell) _{ \ell \in L} $, giving the reduced sequence
$\sigma^r$.  
For each $\ell \in L$, let $J_\ell = 
\{ \, j \in J \mid \ell_1 < I_j <\ell_2 ,
\text{ for some } \ell_1, \ell_2 \in K^*_\ell \,\}$ and let    
 $ K  _\ell  =K^*_\ell \cup \bigcup  _{ j \in J_\ell}  I_{j,\ell}$.
Note that, by construction, $K^*_\ell \cap I_j = \emptyset  $,
for every  $\ell \in L$ and $j \in J$. Hence, for every $\ell \in L$,
$I_{j,\ell}$ ($j \in J_\ell$) is a reducing family for $K  _\ell$.
Since  each $K^*_\ell$ is pseudonull, each $K_\ell$ is pseudonull by   (g2).
Then the family $ \mathcal F$  consisting of the $K_{\ell}$ ($j \in J_\ell$)
together with the $I_j$ ($j \notin \bigcup  _{\ell \in L}J_\ell$)
is a reduction of $I$. This family reduces $\sigma$ to $\sigma^r$,
since $\sigma^*$ is obtained from $\sigma$ by removing the indexes
in $\bigcup  _{ j \in J}  I_{j}$ and then $\sigma^r$ is obtained from
$\sigma^*$  by removing the indexes
in $\bigcup  _{\ell \in L} K^*_\ell$.
Indeed, each $I_j$ is removed in the reduction given by
$ \mathcal F$, since, for every $j \in J$, either $I_j$ is contained  
in some $K_\ell$ (if $i \in J_\ell$), or, otherwise,  $I_j$ directly belongs
to $ \mathcal F$.
\end{proof}

\section{Proof of Theorem \ref{abe}} \labbel{proofabe} 

We now have all the tools at our disposal
for a proof of Theorem \ref{abe}.

We define the domain
$S$ of our desired complete semigroup 
as the set of all the finite strings of the form 
$b^na^m$, with $n,m \geq 0$, plus an additional element $ \Omega$.  
We denote the empty string, that is, the case $n=m=0$ by $e$. 
If $(\sigma_z) _{z \in Z } $ is a sequence of elements of $S$
and $Z$ is a linearly ordered set,
let us define $\prod _{z \in Z}  \sigma _z $ in the following way.

(*) If some $\sigma_z$ is $ \Omega$, we let  
$\prod _{z \in Z}  \sigma _z  = \Omega $.

(**) Otherwise, consider the concatenation
$ \sigma  = \sum _{ z \in Z }  \sigma _z$ of the strings $\sigma_z$
modulo the linearly ordered set $Z$.
We recall the definition  in this particular case.
Suppose that, for every $z \in Z$, $\sigma_z$ has length $n_z$
and is indexed  by the set $\{ \,  0, 1, \dots, n_z-1 \,\}$.
The index set of $\sigma$ is $I= \{ \, (z, \ell ) \mid z \in Z,
\ell < n_z \, \} $, ordered lexicographically.
The element with index $(z, \ell )$ in $\sigma$ 
is defined to be the element with index $\ell $
in $\sigma_z$. 
Having defined $\sigma$, 
we let 
$\prod _{z \in Z}  \sigma _z  = b^na^m $
if $\sigma$ is regular with reduction $b^na^m $, 
in the sense of Definition \ref{fullred}
(this is a good definition in view of Lemma \ref{cl}(f)). 
If $\sigma$ is not regular, let
$\prod _{z \in Z}  \sigma _z  = \Omega $.
 
Having defined $S$, we have to prove Conditions (U) and (N)
from Definition \ref{pnc}.
Condition (U) is elementary, since if  $ \vert   Z  \vert   =1$,
then $\sigma$ is (up to a reindexing) $\sigma_z$
for the unique $z \in Z$; moreover, a sequence of the form 
$b^na^m$ cannot be further
 reduced.

It remains to check that Condition (N) is satisfied.
So let $\sigma$ and $Z$ be as in (**) above and let
$\pi:Z \to J$ be a surjective order preserving map. 
Condition (N) is straightforward if some $\sigma_z$
is $ \Omega$, so let us assume that no $\sigma_z$ is $ \Omega$.
We have to evaluate
$ \prod _{j\in J} \prod _{\pi(z) =j }  \sigma _z$. 
For $j \in J$, let $Z_j = \{ \, z \in Z \mid \pi (z) = j \, \} $,
thus each $Z_j$ is a convex subset of $Z$.
If, for some $j \in J$,   $\prod _{\pi(z) =j }  \sigma _z = \Omega $,
then, by definition,  (the sequence associated to) $Z_j$ is not
regular (recall the convention in Definition \ref{fullsubs}).
By Lemma \ref{cl}(c) in contrapositive form, if 
some $Z_j$ is not regular, then $Z$ (that is, $\sigma$) is not regular,
hence this special instance of (N) holds.

So let us assume that, for every $j \in J$, $Z_j$ is
regular. In particular, each $Z_j$ 
has a reducing family $W_{j,h}$ ($h \in H_j$),
for some index set $H_j$ and subsets   $W_{j,h}$
of $Z_j$, inducing some finite reduced
sequence $\sigma^{*,j}$. In particular,
$ \prod _{j\in J} \prod _{\pi(z) =j }  \sigma _z =
\prod _{j\in J} \sigma^{*,j}$. 
Since each  $Z_j$ is a convex subset of $Z$,
the family  $W_{j,h}$ ($h \in \bigcup _{j \in J} H_j$)
is a reducing family for $Z$, which induces a sequence
$\sigma^*$. The sequence $\sigma^*$ turns out to be the 
string concatenation of the $\sigma^{*,j}$, thus 
$\prod _{j\in J} \sigma^{*,j}$ is the reduction of $\sigma^*$,
if $ \sigma ^*$ is reducible, and $ \Omega$ otherwise.   
We are in the situation dealt with in 
Lemma \ref{cl}(h), so that $\sigma$ is regular if and only if
$\sigma^*$ is regular and, in this case, the regularizing reductions 
are the same. This means that 
$\prod _{z \in Z}  \sigma _z $ and  
$ \prod _{j\in J} \prod _{\pi(z) =j }  \sigma _z =
\prod _{j\in J} \sigma^{*,j}$ give the same outcome. 

The proof of Theorem \ref{abe} is thus complete.


\section{Appendix I. A simpler proof
for a special case of Theorem \ref{abe}}
 \labbel{appabe}

In this appendix we present full details
for the construction of a
${\leq} \omega $-semigroup with  elements $a,b,e$
satisfying the conclusion of Theorem \ref{abe},
as sketched in the example in  Section \ref{XyX}.

We first need an analysis
of finite strings containing only the symbols $a,b$, with the intended idea that
the juxtaposition $ab$ reduces to the identity.
We let a \emph{pseudonull} string be a 
finite string $ \upsilon $ such that 

(p1) $\upsilon$  contains
an equal number of occurrences of $a$ and $b$ and, moreover, 

(p2) for
 any initial substring $\eta$ of $\upsilon$, the number of
occurrences of $a$ in  $\eta$ is $\geq$ than the 
 number of
occurrences of $b$ in  $\eta$.

\noindent Thus, for example,  $aaababbabb$ is a pseudonull string, 
but $baab$ is not pseudonull, since (p2)
is not satisfied. 
From the point of view of (finitary noncommutative) semigroups,
pseudonull strings are strings which can be reduced to $e$
under the assumption $ab=e$, with $e$ a neutral element.

Now consider a finite or $ \omega$-indexed  string $\sigma$  
with only the symbols $a $ and $ b$. The \emph{reduced form}  of 
$\sigma$ is the (possibly empty) string obtained from $\sigma$ by removing  
all the pseudonull substrings (as usual, by a \emph{substring} 
we always mean a convex substring, not simply a subsequence).
 Note that the reduced form of $\sigma$ is a subsequence of $\sigma$,
but not necessarily a substring. Also note that in a reduction 
it might happen that infinitely many elements of the string
are removed, for example, if $\sigma=bba
\overline{aabb} \overline{aabababb} \overline{ab}
 aa \overline{ab} \hspace{1.5 pt} \overline{ab}  \hspace{1.5 pt} \overline{ab} \dots$,
then the reduced form of $\sigma$ is $bbaaa$. 

It is intuitively clear that no occurrence of the substring $ab$
can appear in the reduced
form of $\sigma$, but a proof is needed, since, in principle,
pseudonull substrings could ``amalgamate'' in unexpected ways.
To check the above statement, and also for later use, we will prove 
the following  properties of pseudonull strings.

\smallskip 

($\alpha $) A finite concatenation of pseudonull strings is pseudonull.
If $\nu$ is a pseudonull string, then the concatenation
$a \nu b$ is pseudonull. More generally, if $\sigma$ is a finite string,
$\upsilon$ is a pseudonull substring of $\sigma$ 
and the string  $\sigma \smallsetminus \upsilon  $, the
subsequence of $\sigma$  
obtained  by removing $\upsilon$, is pseudonull,
then $\sigma$ is pseudonull. 

($\beta$) A finite string $ \upsilon $ is pseudonull
if and only if  (p1) holds and, moreover, 

\ \ \ \ \ \ \ \ (p3) for
 any final substring $\eta^*$ of $\upsilon$, the number of
occurrences of $b$ in  $\eta^*$ is $\geq$ than the 
 number of
occurrences of $a$ in  $\eta^*$.

 ($\gamma$) A finite string $ \upsilon $ is pseudonull
if and only if  both (p2) and (p3) above  hold.

($\delta$) If $\upsilon_1,   \upsilon _2$ are pseudonull substrings of $\sigma$, 
then the intersection of  $\upsilon_1 $ and $    \upsilon _2$  is a pseudonull
substring of $\sigma$.

 ($\varepsilon$) If $\upsilon_1,   \upsilon _2$ are pseudonull substrings
 of $\sigma$, 
then both  $\upsilon_1 \smallsetminus   \upsilon _2$  and 
$\upsilon_2 \smallsetminus    \upsilon _1$ are pseudonull strings.

($\zeta$)  If $\sigma$ is a finite union of pseudonull strings,
then $\sigma$ is pseudonull.

($\eta$) If $\sigma$ is a finite or $ \omega$-indexed  string   
with only the symbols $a $ and $ b$, then the reduced form of 
$\sigma$ contains no occurrence of the substring $ab$. In particular,

($\theta$)  if $\sigma$ is reduced, then
$\sigma$ is either an infinite string of $b$s, or has the form
  $b^na^ \alpha $, with $n$ 
a nonnegative integer and $\alpha \leq \omega $,
both $n$ and $ \alpha $ being  
  possibly $0$.

($\iota$)  If $\sigma$ is a finite or $ \omega$-indexed  string,
 the  reduced form of $\sigma$ can be obtained from $\sigma$ by removing
a family of pairwise disjoint pseudonull substrings.

($\kappa$) If $\sigma$ is a finite or $ \omega$-indexed  string and
$\sigma^*$ is obtained from $\sigma$ by removing some
(not necessarily all) pseudonull substrings, then
the reduced forms of $\sigma$ and of $\sigma^*$ are equal.  

\smallskip 

Item ($\alpha$) is straightforward.
To prove ($\beta$), note that any final substring 
of $\upsilon$ is the complement of some
 initial substring of $\upsilon$,
so that, assuming (p1), the conditions (p2) and (p3)
are equivalent.
Then ($\gamma$) follows, since, when applied to the improper substring,
 (p2) and (p3) together imply (p1). 

Item ($\delta$) is straightforward if the two strings
have empty intersection, or if
$\upsilon_1 $ is contained in $    \upsilon _2$,
or conversely. Otherwise,
 suppose, say, that the first element of 
$\upsilon_1 $ precedes the first element of $    \upsilon _2$.
Then their intersection $\upsilon$ is 
an initial segment of $   \upsilon _2$, hence 
$\upsilon$ satisfies (p2), since $   \upsilon _2$ 
satisfies (p2).
Symmetrically, 
$\upsilon$ is 
a final segment of $   \upsilon _1$, hence 
$\upsilon$ satisfies (p3).
The conclusion follows from ($\gamma$).

Item ($\varepsilon$) then follows, since 
$\upsilon_1 \smallsetminus   \upsilon _2 $ is equal to 
$\upsilon_1 \smallsetminus   \upsilon $, where, as above,  $\upsilon$
is the intersection 
of $\upsilon_1 $ and  $    \upsilon _2$. Indeed,
the sequence $\upsilon_1 \smallsetminus   \upsilon $ satisfies (p1) since both
$\upsilon_1 $ and $   \upsilon $
satisfy (p1). If the first element of 
$\upsilon_1 $ precedes the first element of $    \upsilon _2$,
then $\upsilon_1 \smallsetminus   \upsilon $
satisfies (p2). All the other cases are symmetrical,
possibly using (p3) and ($\beta$). 

 Note that if, say,
$\upsilon_2 $ is contained in  $    \upsilon _1$,
then $\upsilon_1 \smallsetminus   \upsilon _2$ is a subsequence 
of $\sigma$, but not 
necessarily a substring of $\sigma$.
On the other hand, if $\upsilon_1 $ and  $    \upsilon _2$
are substrings of $\sigma$, then their intersection is still
a (possibly empty) substring of $\sigma$.

Item ($\zeta$) is proved by iterating ($\delta$) and ($\varepsilon$).
We get that $\sigma$ is the disjoint union of pseudonull
substrings, so that $\sigma$ is pseudonull by the first
statement in ($\alpha$).  
Note that, since $\sigma$ is supposed to be a \emph{finite}
union of substrings, it is no loss of generality to assume that
such substrings are pairwise containment-incomparable
(since otherwise we can remove any string contained in some
larger string, still getting the same union). 
The difference of two containment-incomparable substrings
is either empty, or is itself a substring, hence we can actually
always work with substrings. 

To prove ($\eta$), suppose by contradiction that $ab$ is a substring of
the reduced form of $\sigma$. Then  $a \nu b$ is a substring of
$\sigma$, where $\nu$ is a  union of pseudonull strings.
Note that $\nu$ is necessarily finite, since $\sigma$ is either finite,
or $ \omega$-indexed, hence the occurrence of 
$b$ has finite index. By ($\zeta$), $\nu$ is pseudonull.
But then also $a \nu b$ would be pseudonull, by ($\alpha$).
Thus the corresponding occurrences of $a$ and $b$
are eliminated in the reduction, a contradiction.
Item ($\theta$) then follows immediately. 

If $\sigma$ is finite, item ($\iota$) follows from
($\delta$) and ($\varepsilon$), as in the proof of ($\zeta$).
Otherwise, let us say that some  sequence $( \nu_p) _{ p \in P} $  of
pseudonull substrings of $\sigma$ is a \emph{reducing sequence} if  
$\sigma \smallsetminus  \bigcup_{ p \in P}  \nu_p  $ is the reduced form
of $\sigma$. Given such a reducing sequence, we will construct 
inductively another reducing  sequence $( \xi _n ) _{ n \in Q } $
of pairwise disjoint pseudonull strings,
with either $Q$ finite or $Q= \omega $. 
Obviously, if $\sigma$ is already reduced, we need just
 consider an empty family.
Otherwise, consider the first element $x_{i_1}$ of
$\sigma$ which does not appear in the reduced form.
By the definition of a reduced form, 
$x_{i_1}$ appears in some $\nu_{p_1}$ with $p_1 \in P$.
Let $ \nu^1_{p_1} =  \nu_{p_1} $ and, 
for $p \in P \setminus \{ p_1 \} $,
 let  $ \nu^1_p =  \nu _p \smallsetminus \nu_{p_1}$.
Since  $x_{i_1}$ is the first element removed in the reduction
process, for every $p \in P$,  $x_{i_1}$ is 
either the first element of $\nu_{p}$, or
 $x_{i_1}$ precedes every element of 
$\nu_p$. Since $x_{i_1}$ is the first element of $\nu_{p_1}$,
 we get  that, for each $p \in P \setminus \{ p_1 \} $,
$ \nu^1_p =  \nu _p \smallsetminus \nu_{p_1}$ is 
a (possibly empty) substring of 
$\sigma$. Moreover, each $\nu^1_p$ is pseudonull by ($\varepsilon$).
It follows that $( \nu^1_p) _{ p \in P} $ is still
a reducing sequence and, moreover, $ \nu^1_{p_1}$
is disjoint with the remaining strings $\nu^1_p$ (for $  p \in P 
\setminus \{ p_1 \}$). 

Now consider the first element $x_{i_2}$ which
 is removed in the reduction
but does not belong to $ \nu^1_{p_1}$
(if there is no such element, the one element sequence consisting
only of $ \nu^1_{p_1}$ is already reducing).
Since $x_{i_2}$ is removed, $x_{i_2}$ appears in some $\nu^1_{p_2}$,
for $p \in P \smallsetminus \{ p_1 \} $.
Letting $ \nu^2_{p_h} =  \nu^1_{p_h} $, for $h=1,2$  and
  $ \nu^2_p =  \nu^1 _p \smallsetminus \nu^1_{p_2}$,
for $p \in P \setminus \{ p_1, p_2 \} $,
arguing as above, we get that $( \nu^2_p) _{ p \in P} $ is still
a reducing sequence, $ \nu^2_{p_1}$ and $ \nu^2_{p_2}$
are pairwise disjoint and also
both disjoint with the remaining strings.
Going on in the same way, we get a possibly infinite
sequence $ \nu^2_{p_1},  \nu^2_{p_2},  \nu^3_{p_3}, \dots $
which is reducing and whose elements are pairwise disjoint.

Finally, we prove ($\kappa$).
By ($\iota$), $\sigma$ has a finite or 
countable reducing family  
$( \xi _n ) _{ n \in Q } $
consisting of pairwise disjoint pseudonull substrings.
Let $( \nu_p) _{ p \in P} $ be the set of pseudonull substrings 
which are removed from $\sigma$ in order to obtain $\sigma^*$.
By the definitions of reduced form and of a 
reducing family, each $\nu_p$ is contained 
in  $\bigcup_{ n \in Q}  \xi_n$.
Without loss of generality, we can also assume
that   each $\nu_p$ is contained in some 
$ \xi_n$, since, if this is not the case, we can consider
the family  $( \nu_p \cap \xi_n) _{ p \in P, n \in Q} $
 in place of $( \nu_p) _{ p \in P} $.
Note that the intersection of two pseudonull 
substrings of $\sigma$ is still a pseudonull substring 
of $\sigma$, by ($ \delta $).

Now observe that each $ \xi_n$ is finite, being a pseudonull string,
hence each $ \xi_n$ contains a finite number 
of $\nu_p$ (without loss of generality we can discard repeated 
occurrences). Fix some $n \in Q$. By a finite iteration of  ($\varepsilon$),
if we remove from  $ \xi_n$ all the corresponding 
$\nu_p$, we get a pseudonull string $ \xi^*_n$.
This is not necessarily a substring of $\sigma$,
but it is indeed  a substring of $\sigma^*$,
since the $\nu_p$ do not belong $\sigma^*$,
But then $ \xi^*_n$ is removed, when reducing
$\sigma^*$. Since the above argument applies
to each $n \in Q$, we get that  the reduced forms
of $\sigma$ and of $\sigma^*$ coincide.

\smallskip

Having proved ($\alpha$) - ($\kappa$), we are now ready to define
a ${\leq} \omega $-semigroup satisfying the desired properties.

The domain
$S$ is the set of all the finite strings of the form 
$b^na^m$, with $n,m \geq 0$, plus an additional element $ \Omega$.  
We denote the empty string, that is, the case $n=m=0$ by $e$. 
If $(\sigma_i) _{i \in I } $ is a sequence of elements of $S$
with either $I= \omega $, or $I$ a nonnegative integer,
let us compute $\prod _{i \in I}  \sigma _i $ in the following way.

If some $\sigma_i$ is $ \Omega$, we let  $\prod _{i \in I}  \sigma _i  = \Omega $.

Otherwise, consider the concatenation
$ \sigma  = \ast _{ i \in I }  \sigma _i$ of the strings $\sigma_i$,
in their order, and let $\tau$ be the reduced form of $\sigma$, as defined above.
If $\tau$  has the form $b^na^m$, with $ n,m < \omega$  let
$\prod _{i \in I}  \sigma _i  = b^na^m$, otherwise, let 
$\prod _{i \in I}  \sigma _i  = \Omega $.
The above definition includes the case $\prod_ \emptyset = e
=b^0a^0$. 
 
Condition (U) is elementary, since any string $b^na^m$ is  
already in reduced form.

It remains to check that Condition (N) is satisfied.
Condition (N) is straightforward if some $\sigma_i$
is $ \Omega$, so let us assume that no $\sigma_i$ is $ \Omega$.
For $I$ as above, we have to evaluate
$ \prod _{j\in J} \prod _{\pi(i) =j }  \sigma _i$, for 
$\pi:I \to J$  a surjective order preserving map.
For each $j \in J$, let $\ast _{\pi(i) =j }  \sigma _i$ denote the
concatenation of the finite or countable set of adjacent
strings such that $\pi(i) =j$. In particular,
$\ast _{\pi(i) =j }  \sigma _i$ is a  substring of $\ast _{ i \in I }  \sigma _i$.
Then, for every $j \in J$, $\prod _{\pi(i) =j }  \sigma _i$ is obtained by reducing  
 $\ast _{\pi(i) =j }  \sigma _i$. 

First, assume that, for every $j \in J$, $\prod _{\pi(i) =j }  \sigma _i
\neq \Omega $ and let $\sigma^*_j = \prod _{\pi(i) =j }  \sigma _i$.
Then  $ \prod _{j\in J} \prod _{\pi(i) =j }  \sigma _i$
is evaluated considering the reduction of $\sigma^* = 
\ast _{j \in J }  \sigma ^*_j$.
We are exactly in the situation from ($\kappa$), hence
the reduced forms of $\sigma$ and of $\sigma^*$ are equal. 
This means that   
$\prod _{i \in I}  \sigma _i  =\prod _{j\in J} \prod _{\pi(i) =j }  \sigma _i$,
with possible value $ \Omega$. 

It remains to treat the case when $\prod _{\pi(i) =j }  \sigma _i
= \Omega $, for some  $\bar {\jmath} \in J$. Because of ($\theta$),
this may happen only when $\{ \, i \in I  \mid \pi(i) = \bar{\jmath}  \,\}$
is infinite, thus $I= \omega $,   $\{ \, i \in I  \mid \pi(i) = \bar{\jmath}  \,\}
 = [n, \omega )$,
for some $n < \omega $, and the reduced form of 
  $\ast _{\pi(i) = \bar{\jmath}  }  $ contains either infinitely many
consecutive $a$ or infinitely many 
consecutive  $b$. As above, letting $\sigma^*_j$ be the reduced
form of $\ast _{\pi(i) =j }  \sigma _i$, for every $j \in J$, and letting   
$\sigma^* = \ast _{j \in J }  \sigma ^*_j$, by ($\kappa$) 
the reduced forms of $\sigma$ and of $\sigma^*$ are equal.
As mentioned above, $\sigma^*$ has an infinite tail consisting
of the same repeating element.
Since $\sigma^*_{ \bar{\jmath} }$ is already reduced,
 since every pseudonull string is finite and also
$\{ \, i \in I  \mid \pi(i) < \bar{\jmath}  \,\}$ is finite,
only a finite number of elements of $\sigma^*$
are removed in the reduction. Since
the reduced forms of $\sigma$ and of $\sigma^*$ are equal,
also the reduced form of $\sigma$
has an infinite tail consisting
of the same repeating element, and this implies that
$\prod _{i \in I}  \sigma _i = \Omega  $, thus 
$\prod _{i \in I}  \sigma _i  =\prod _{j\in J} \prod _{\pi(i) =j }  \sigma _i$
in this final case, as well.

\end{document}